\documentclass[letterpaper, 10pt]{IEEEtran}

\usepackage{mathtools,dsfont,amsfonts,amssymb}

\usepackage{enumitem}
\usepackage{subcaption}

\usepackage[amsmath,thmmarks]{ntheorem}

\theoremstyle{plain}
\theorembodyfont{\itshape}
\theoremheaderfont{\normalfont\bfseries}
\newtheorem{theorem}{Theorem}
\newtheorem{lemma}{Lemma}
\newtheorem{corollary}{Corollary}

\theorembodyfont{\upshape}
\newtheorem{definition}{Definition}
\newtheorem{remark}{Remark}

\theoremstyle{nonumberplain}
\theoremheaderfont{\scshape}
\theorembodyfont{\normalfont}
\theoremsymbol{\ensuremath{\blacksquare}}
\newtheorem{proof}{Proof}

\theoremstyle{plain}

\theorembodyfont{\upshape}
\newtheorem{assumption}{Assumption}

\usepackage[svgnames]{xcolor}
\usepackage{hyperref}
\hypersetup{
    colorlinks=true,
    linkcolor=blue,
    filecolor=magenta,      
    urlcolor=cyan,
    citecolor=Green,
    }
\usepackage{booktabs}
\usepackage{subcaption}
\usepackage{graphicx}
\usepackage[sort,compress]{cite}
\usepackage{soul}
\setstcolor{red}
\newcommand*\ALPHABET{\mathcal}
\newcommand*\ABSTRACT{\tilde}
\newcommand*\ESTIMATE{\hat}

\newcommand*\B[1]{\ALPHABET B(\ALPHABET #1)}

\newcommand*\PR{\mathds{P}}
\newcommand*\EXP{\mathds{E}}
\newcommand*\IND{\mathds{1}}
\newcommand*\reals{\mathds{R}}

\newcommand{\TRANS}{\intercal}

\newcommand\dWas{d_{\mathrm{Was}}}
\newcommand\dS{d_{\ALPHABET S}}
\newcommand\dhS{d_{\ABSTRACT {\ALPHABET S}}}
\def\dhSphi(#1, #2){d_{\ABSTRACT {\ALPHABET S}}(\phi(#1), \phi(#2))}
\newcommand\EST{\mathcal{E}}

\DeclareMathOperator{\Lip}{Lip}

\newcommand*\AIS{\mathsf{AIS}}
\newcommand*\CE{\mathsf{CE}}

\begin{document}
\title{Sub-optimality bounds for certainty equivalent policies in partially observed systems}
\author{Berk Bozkurt, Aditya Mahajan, Ashutosh Nayyar, and Yi Ouyang%
\thanks{A preliminary version of this work appeared in~\cite{bozkurtgeneralized}.}
\thanks{Berk Bozkurt is with INLAN, Montreal, QC, Canada. (email: berk.bozkurt@mail.mcgill.ca).} 
\thanks{Aditya Mahajan is with the Department of Electrical and Computer Engineering, McGill University, Montreal, QC, Canada. (email: aditya.mahajan@mcgill.ca).}%
\thanks{Ashutosh Nayyar is with the Department of Electrical and Computer Engineering, University of Southern California, Los Angeles, CA, USA. (email: ashutosn@usc.edu).}%
\thanks{Yi Ouyang is with Atmanity, Santa Clara, CA, USA. (email: ouyangyii@gmail.com).}}
\maketitle 

\begin{abstract}
    In this paper, we present a generalization of the certainty equivalence principle of stochastic control. One interpretation of the classical certainty equivalence principle for linear systems with output feedback and quadratic costs is as follows: the optimal action at each time is obtained by evaluating the optimal state-feedback policy of the stochastic linear system at the minimum mean square error (MMSE) estimate of the state. Motivated by this interpretation, we consider certainty equivalent policies for general (non-linear) partially observed stochastic systems that allow for any state estimate rather than restricting to MMSE estimates. In such settings, the certainty equivalent policy is not optimal. For models where the cost and the dynamics are smooth in an appropriate sense, we derive upper bounds on the sub-optimality of certainty equivalent policies. We present several examples to illustrate the results.
\end{abstract}

\section{Introduction}
In many applications in robotics, autonomous systems, finance, healthcare, and other domains the decision maker does not have access to the complete state of the system. Such systems are often modeled as partially observable Markov decision processes (POMDPs). The standard approach for solving POMDPs is to translate them into fully observed Markov decision processes (MDPs) by considering the posterior belief of the decision maker on the current state as a sufficient statistic~\cite{Astrom1965,Smallwood1973}. There is a rich literature on algorithms which use the structure of the resulting belief-state MDP to obtain optimal and approximately optimal policies.

However, finding the optimal policy is PSPACE-hard~\cite{papadimitriou1987complexity}. Most algorithms to find optimal policies have exponential worst-case complexity in the size of the state and observation spaces making them impractical for large-scale problems. Meanwhile, many heuristic approaches such as point-based value iteration methods~\cite{shani2013survey} can be computationally efficient but lack provable performance guarantees. In fact, finding approximately optimal policies is also PSPACE-hard~\cite{burago1996complexity,lusena2001nonapproximability}, indicating that general purpose algorithms may not be efficient for all POMDP models.

These challenges have motivated significant interest in identifying structured classes of policies that are both computationally tractable and have good performance guarantees. Examples include policies based on a finite window of previous observations~\cite{white1994finite} (called frame stacking in reinforcement learning) and, more generally, policies based on \emph{agent state}, which is a recursively updatable function of the past observations and actions~\cite{vanroy2021simple,sinha2024agent}. 
Several papers have identified sufficient conditions for such policies to perform well, including approximate information state~\cite{AIS}, filter stability~\cite{mcdonald2022robustness,kara2022near,golowich}, weakly revealing observations~\cite{liu2022partially}, low covering numbers~\cite{lee2007makes}, low-rank structure~\cite{guo2023provably}, and revealing observation models~\cite{belly2025revelations}.  These results highlight that structured policies can be approximately optimal for a specific sub-class of POMDPs. 

In linear systems with quadratic cost and Gaussian noise (the so-called LQG problem), the optimal policy may be viewed as a structured policy with the following structure: the optimal action at each time is a linear function of the MMSE (minimum mean square error) state estimate and the corresponding feedback gain is the same as the feedback gain of the optimal \emph{state-feedback} control of the \emph{deterministic} system obtained by replacing all random variables by their means. This result is typically called the \emph{certainty equivalence principle of stochastic control}~\cite{Theil1954,Theil1957,simon1956dynamic} and has been generalized to other settings including systems with non-linear dynamics~\cite{bar1974dual,derpich2022dual,li2025semilinear} and risk-sensitive control~\cite{whittle1986risk}.

In this paper, we present a generalization of the certainty equivalence principle to general POMDPs. Our generalization is based on a slightly different interpretation of the LQG certainty equivalence principle: consider the optimal policy of the \emph{stochastic} system with perfectly observable states and evaluate that policy at the MMSE state estimate. Similar views on the certainty equivalence principle have been used in the reinforcement learning and adaptive control literature~\cite{hardtrecht2022patterns} and are sometimes called partially stochastic certainty equivalence~\cite{bertsekas2012dynamic}. For clarity, we present a formal description of this interpretation.

Let $\ALPHABET P$ denote the partially observable linear system with state $s_t \in \ALPHABET S$, action $a_t \in \ALPHABET A$, and output $y_t \in \ALPHABET Y$, where $\ALPHABET S$, $\ALPHABET A$, and $\ALPHABET Y$ are Euclidean spaces. Let $\ALPHABET M$ be the fully observable stochastic linear system where the decision maker has access to the state. Note that the fully observed system $\ALPHABET M$ is different from one typically assumed in certainty equivalence. As is the case in the standard certainty equivalence principle, we are assuming that $\ALPHABET M$ is fully observed but we are not assuming that the dynamics of $\ALPHABET M$ are deterministic. For simplicity, suppose that the system runs for a finite horizon $T$. Let $\pi^{\ALPHABET M} = (\pi^{\ALPHABET M}_1, \dots, \pi^{\ALPHABET M}_T)$ denote the optimal policy for model $\ALPHABET M$ and $\mu^{\ALPHABET P} = (\mu^{\ALPHABET P}_1, \dots, \mu^{\ALPHABET P}_T)$ denote the optimal policy for model~$\ALPHABET P$. Moreover, for any history $h_t = (y_1, a_1, y_2, a_2, \dots, y_{t})$ of observations and actions until time $t$, let $\EST_t(h_t)$ denote the MMSE estimate of the state given the history $h_t$. Then, the standard result for LQG optimal control is that
\[
    \mu^{\ALPHABET P}_t(h_t) = \pi^{\ALPHABET M}_t(\EST_t(h_t)).
\]

In this paper, we consider two generalizations of the above result.
\begin{enumerate}
    \item We allow $\EST_t$ to be \emph{any} estimator of the state rather than restricting attention to MMSE estimator.
    \item We consider general POMDPs rather than restricting attention to linear systems.
\end{enumerate}
In this general setting, we define the certainty equivalent policy $\mu^{\CE} = (\mu^{\CE}_1, \dots, \mu^{\CE}_T)$ as
\begin{equation}\label{eq:CE}
    \mu^{\CE}_t(h_t) = \pi^{\ALPHABET M}_t(\EST_t(h_t)).
\end{equation}
In general, $\mu^{\CE}$ is not optimal. Our main result is to characterize the degree of sub-optimality of the certainty equivalent policy $\mu^{\CE}$.

Our results may be viewed as an instance of characterizing the sub-optimality gap of structured policies for POMDPs. There is a rich literature on deriving such sub-optimality gaps using tools from predictive state representation~\cite{Wolfe2008,Hamilton2014}, bisimulation metrics~\cite{CastroPanangadenPrecup_2009}, approximation information states (AIS)~\cite{AIS}, and filter stability~\cite{mcdonald2022robustness,kara2022near,golowich}. Our analysis is based on AIS-based approximation bounds of~\cite{AIS}. 

The main contributions of this paper are as follows:
\begin{itemize}
\item  We derive explicit bounds on the sub-optimality of certainty equivalent policies under the assumption that the cost and dynamics are smooth in an appropriate sense. Our bounds depend on the worst-case value of the conditional expected estimation error. 

\item We extend our results to settings with state abstraction, where the estimator produces an estimate of an abstract state rather than the full state, allowing our framework to apply to large-scale systems where state aggregation/quantization or feature abstraction is necessary.

\item  We illustrate our results through several examples, including: systems with bounded observation noise, intermittently degraded observations, control with event-triggered communication, learning and adaptive control settings, and control of non-homogeneous multi-particle systems. These examples demonstrate that certainty equivalent policies can achieve near-optimal performance when the estimation error is small, providing practical and computationally tractable alternatives to exact POMDP solutions.
\end{itemize}

A preliminary version of this result appeared in~\cite{bozkurtgeneralized}. The analysis there was restricted to certainty equivalent policies which estimate the complete state of the system and the results were obtained under stronger assumptions. The state abstraction model considered in this paper is new and the results are derived under weaker assumptions.

The rest of the paper is organized as follows. We  present the system model and define certainty equivalent policies in Sec.~\ref{sec:system_model}. We illustrate our results through several examples in Sec.~\ref{sec:examples}. We present the proofs in Sec.~\ref{sec:analysis} and conclude in Sec.~\ref{sec:conclusion}.

\paragraph*{Notation} We use uppercase letters to denote random variables (e.g., $S$, $A$, etc.), the corresponding lowercase letters to denote their realizations (e.g., $s$, $a$, etc.), and the corresponding calligraphic letters to denote their space of realizations (e.g., $\ALPHABET S$, $\ALPHABET A$, etc.). Subscripts denote time, so $S_t$ denotes a variable at time $t$. The notation $S_{1:t}$ is a shorthand for the sequence $(S_1, \dots, S_t)$. 

We use $\reals$ to denote the set of real numbers. For a topological space $\ALPHABET X$, $\Delta(\ALPHABET X)$ denotes the set of all probability measures on $\ALPHABET X$ and $\B{X}$ denotes the set of all bounded and measurable real-valued functions on $\ALPHABET X$. 

We use $\PR(\cdot)$ to denote the probability of an event and $\EXP[\cdot]$ to denote the expectation of a random variable. We use the notation $\PR(S_{t+1} \in M_S | s_t, a_t)$ as a shorthand for $\PR(S_{t+1} \in M_S | S_{t} = s_t, A_t = a_t)$. 

Given a metric space $(\ALPHABET S, \dS)$, the Wasserstein-1 distance between two probability distributions $\nu_1, \nu_2 \in \Delta(\ALPHABET S)$ is given by
\[
   \dWas(\nu_1,\nu_2) = \inf_{(S_1, S_2) \sim \Gamma(\nu_1, \nu_2)}\EXP[\dS(S_1, S_2)],
\]
where $\Gamma(\nu_1, \nu_2)$ denotes all joint probability distributions on $\ALPHABET S \times \ALPHABET S$ with marginals $\nu_1$ and $\nu_2$. For two  random variables $S_1$ and $S_2$ taking values in $\ALPHABET S$, we sometimes use $\dWas(S_1, S_2)$ to denote the Wasserstein-1 distance between the marginal distributions of $S_1$ and $S_2$. For a function $f$ from one metric space to another, $\Lip(f)$ denotes the Lipschitz constant of $f$.

\section{System model and the main results}\label{sec:system_model}

Consider a discrete-time partially observable Markov decision process (POMDP), denoted by $\ALPHABET P$, with state space $\ALPHABET S$, observation space $\ALPHABET Y$, and action space $\ALPHABET A$ that runs for a finite horizon~$T$. Let $S_t \in \ALPHABET S$ denote the state of the system, $Y_t \in \ALPHABET Y$ denote the observation of the controller, and $A_t \in \ALPHABET A$ denote the control action taken by the controller at time~$t$. We assume that $\ALPHABET S, \ALPHABET Y$ and $\ALPHABET A$ are metric spaces and use $\dS$ to denote the metric on $\ALPHABET S$.

The initial state and observation $(S_1, Y_1)$ are distributed according to a probability distribution $\xi \in \Delta(\ALPHABET S \times \ALPHABET Y)$. The dynamics and observation are assumed to be Markovian.
In particular, we assume that there exist stochastic kernels $P_t \colon \ALPHABET S \times \ALPHABET A \to \Delta(\ALPHABET S \times \ALPHABET Y)$, $t \in \{1, \dots, T-1\}$, such that for any $t \in \{1, \dots, T-1\}$, any Borel subsets $M_S, M_Y$ of $\ALPHABET S$ and $\ALPHABET Y$ respectively, and any realizations $s_{1:t}$, $y_{1:t}$ and $a_{1:t}$ of $S_{1:t}$, $Y_{1:t}$, $A_{1:t}$, respectively, we have
\begin{align}
    \hskip 1em & \hskip -1em
    \PR(S_{t+1} \in M_S, Y_{t+1} \in M_Y | s_{1:t}, y_{1:t}, a_{1:t})
    \notag \\
    &=
    \PR(S_{t+1} \in M_S, Y_{t+1} \in M_Y | s_{t}, a_{t})
    \notag\\
    &\eqqcolon P_t(M_S, M_Y | s_t,a_t).
\end{align}
We will use the notation $P_{S,t}(\cdot | s_t, a_t)$ and $P_{Y,t}(\cdot | s_t,a_t)$ to denote the state and observation marginals of $P_t(\cdot, \cdot | s_t,a_t)$. 

At each time~$t$, the system incurs a per-step cost $c_t(S_t,A_t)$, which is uniformly bounded i.e., there exists a $c_{\max} \in \reals$ such that $\sup_{s \in \ALPHABET S, a \in \ALPHABET A} |c_t(s,a)| \le c_{\max}$.

The controller has access to observation and action history $h_t = \{y_{1:t}, a_{1:t-1}\}$ at time~$t$. Let $\ALPHABET H_t$ denote the space of realizations of $h_t$. Let $\mu = (\mu_1, \dots, \mu_T)$ denote any history dependent deterministic policy. The value function of policy $\mu$ is defined as
\[
    W^{\ALPHABET P, \mu}_t(h_t) = \EXP^{\mu}\biggl[ \sum_{\tau = t}^T c_{\tau}(s_{\tau}, a_{\tau}) \biggm| h_t \biggr]
\]
where $\EXP^{\mu}$ denotes expectation with respect to a joint probability measure on the system variables induced by the policy~$\mu$.
The \emph{optimal} value function is defined as
\[
    W^{\ALPHABET  P}_t(h_t) = \inf_{\mu} W^{\ALPHABET P, \mu}_t(h_t),
\]
where the infimum is over all history dependent policies. 

The standard approach to find an optimal policy in POMDPs is to use belief-state based dynamic programs~\cite{Astrom1965,Smallwood1973}, which are computationally challenging. As discussed in the introduction, certainty equivalent policies provide an attractive alternative approach. In the rest of this section, we characterize the sub-optimality of such policies. 

\subsection{Certainty equivalent policies}\label{sec:gce}
Consider a state feedback controller for the stochastic system defined above, where the controller has access to the state $S_t$ at time~$t$. This system is a finite horizon Markov decision process (MDP) $\ALPHABET M$ with state space $\ALPHABET S$, action space $\ALPHABET A$, dynamics $P_{S,t}$, and per-step cost $c_t$. 

We need a technical assumption to ensure that the MDP $\ALPHABET M$ has an optimal policy. 
\begin{definition}[Measurable selection]
    An MDP $\langle \ALPHABET S$, $\ALPHABET A$, $\{P_{S,t}\}_{t=1}^{T-1}$, $\{c_t\}_{t=1}^T, T\rangle$
    is said to satisfy \emph{measurable selection} if for every measurable function $V\colon \ALPHABET S \to \reals$ and each time $t \in \{1, \dots, T-1\}$, there exists a measurable selector $\pi\colon \ALPHABET S \to \ALPHABET A$ such that 
        \begin{align*}
        & \inf_{a \in \ALPHABET A} \biggl\{
        c_t(s, a) + \int_{\ALPHABET S} V(s') P_{S,t}(ds' | s,a) 
        \biggr\}
        \\
        & = 
        c_t(s, \pi(s)) + \int_{\ALPHABET S} V(s') P_{S,t}(ds' | s,\pi(s))
        \eqqcolon V_{+}(s),
                \end{align*}
        and $V_{+}\colon \ALPHABET S \to \reals$ defined above is a measurable function.
\end{definition}

\begin{assumption}\label{ass:meas-selection}
    The model ${\ALPHABET M}$ satisfies measurable selection.
\end{assumption}

An implication of ${\ALPHABET M}$ satisfying measurable selection is that there exists an optimal policy $\pi^{{\ALPHABET M}} = (\pi^{{\ALPHABET M}}_1, \dots, \pi^{{\ALPHABET M}}_T)$, where $\pi^{ {\ALPHABET M}}_t \colon  {\ALPHABET S}  \to \ALPHABET A$, with associated optimal value functions $(V^{ {\ALPHABET M}}_1, \dots, V^{ {\ALPHABET M}}_T)$, $V^{ {\ALPHABET M}}_t \colon  {\ALPHABET S} \to \reals$, for this MDP~\cite{hernandez2012discrete}. 

We now use the optimal policy $\pi^{ {\ALPHABET M}}$ for the MDP~${\ALPHABET M}$ to define a feasible policy for the POMDP~$\ALPHABET P$. 
Suppose we are given a sequence of \emph{ state estimation functions} $\{\EST_t\}_{t=1}^T$, where $\EST_t \colon \ALPHABET H_t \to  {\ALPHABET S}$. 
For instance, $\EST_t$ may be the conditional mean or the MAP (maximum a posteriori probability) estimator which depend on the conditional distribution of the state given the history of observations and actions. Alternatively, the estimator could be a simple function (e.g. linear) of the last few  observations and actions. 

We say that a history-dependent policy $\mu^\EST = (\mu^\EST_1, \dots, \mu^\EST_T)$ is \emph{certainty equivalent with respect to MDP $ {\ALPHABET M} = \langle \ALPHABET S, \ALPHABET A, \{P_{S,t}\}_{t=1}^{T-1}, \{c_t\}_{t=1}^T,T\rangle$ and estimators  $\{\EST_t\colon \ALPHABET H_t \to  {\ALPHABET S}\}_{t \ge 1}$} if 
\begin{align}\label{eq:cert_equiv}
    \mu^\EST_t(h_t) = \pi^{ {\ALPHABET M}}_t(\EST_t(h_t)).
\end{align}
In other words, a certainty equivalent policy treats   $\EST_t(h_t)$ as an error-free estimate of the  state of the  MDP $ {\ALPHABET M}$ and then acts according to the optimal policy of $ {\ALPHABET M}$.
As mentioned earlier, such policies are optimal in the LQG setting when the conditional mean is used as the estimate but they are, in general, not optimal. We are interested in providing a bound on the sub-optimality of the certainty equivalent policies. Specifically, we are interested in an upper bound on the gap between the value functions of policy $\mu^\EST$ and the optimal value functions of the POMDP, i.e. a bound on $W^{{\ALPHABET{P}},\mu^\EST}_t(h_t) - W^{\ALPHABET{P}}_t(h_t)$. Our results provide such a bound under the following technical assumption on the ``smoothness'' of per-step cost and system dynamics. 

 \begin{assumption}\label{ass:lipschitz} 
    There exist a sequence of concave and non-decreasing functions $F^P_t, F^c_t \colon \reals_{\ge 0} \to \reals_{\ge 0}$, $t \in \{1, \dots, T\}$, such that for any $s, s' \in \ALPHABET S$ and $a \in \ALPHABET A$, we have
    \begin{equation}\label{eq:Lip-P-1}
        \dWas(P_{S,t}(\cdot | s, a), P_{S,t}(\cdot | s', a)) \le F^P_t(d_{\ALPHABET S}(s, s'))
    \end{equation}
    and
    \begin{equation}\label{eq:Lip-c-1}
        \bigl| c_t(s,a) - c_t(s',a) \bigr| \le F^c_t(d_{\ALPHABET S}(s, s')).
    \end{equation}
\end{assumption}
\begin{remark}\label{rem:lip_linear}
When $F^P_t$, $F^c_t$ are linear, i.e., $F^P_t(x) = L_t^{P}x$ and $F^c_t(x) = L_t^{c}x$ for some positive constants $L^P_t$ and $L^c_t$, then 
Assumption \ref{ass:lipschitz} reduces to assuming that the dynamics and per-step cost are Lipschitz continuous, which  is a  standard assumption for smoothness of the dynamics and per-step cost, and implies smoothness (Lipschitz continuity) of the value function of MDP~$ {\ALPHABET M}$~\cite{Hinderer2005}. In particular, following the argument of~\cite{Hinderer2005}, we have
    \[
        \Lip(V^{{\ALPHABET M}}_t) \le L^c_t + L^P_{t+1} \Lip(V^{{\ALPHABET M}}_{t+1}),
    \]
    which can be unrolled to obtain an upper bound on the Lipschitz constant of $V^{\ALPHABET M}_t$ in terms of $\{L^P_{\tau}\}_{\tau=t+1}^{T-1}$ and $\{L^c_{\tau}\}_{\tau=t}^{T}$.
\end{remark}

 Our bounds for the sub-optimality gap of certainty equivalent policy $\mu^\EST_t(h_t)$ defined in \eqref{eq:cert_equiv}
 depend on the {quality} of the estimates produced by the state estimation functions $\EST_t$, which we assess using the metric $d_{{\ALPHABET S}}$ on the state space. For each time $t$, we define
\begin{align}\label{eq:closeness_assm}
      \eta_t &:= \sup_{h_t \in \ALPHABET H_t} \EXP[ d_{\ALPHABET S}(S_t, \EST_t(h_t)) |  h_t ].
    \end{align}
    
\begin{remark}\label{rem:policy_indep}
Note that in Eq.~\eqref{eq:closeness_assm} the right hand side does not depend on the policy because the conditional probability distribution of current state $S_t$ given history $h_t$ is policy independent~\cite{Astrom1965,Smallwood1973}. 
\end{remark}

\begin{assumption}\label{ass:finite_eta}
For $t=1,2,\ldots,T$ we have $\eta_t < \infty$ where $\eta_t$ is given by~\eqref{eq:closeness_assm}.
\end{assumption}

We can now state our first result.
\begin{theorem}\label{thm:main_theorem}
Define
\[
    \varepsilon_t = F^c_t (\eta_t) \text{ and } \delta_t = F^P_t (\eta_t) + \eta_{t+1}
\]
   where $\eta_t$ is given by \eqref{eq:closeness_assm}.
    Then, under Assumptions~\ref{ass:meas-selection}, \ref{ass:lipschitz} and \ref{ass:finite_eta}, we have that the certainty equivalent policy $\mu^\EST$ (defined in~\eqref{eq:cert_equiv}) satisfies 
    \begin{equation}\label{eq:main_theorem_rep_ex1}
        W^{\ALPHABET P, \mu^\EST}_t(h_t) - W^{\ALPHABET P}_t(h_t) \le 2 \alpha_t
    \end{equation}
    where
    \begin{equation}\label{eq:alpha}
        \alpha_t = \varepsilon_t + \sum_{\tau = t}^{T-1}\bigl[ \delta_{\tau}\Lip(V^{{\ALPHABET M}}_{\tau+1}) + \varepsilon_{\tau+1} \bigr]
    \end{equation}
    and $\{V^{{\ALPHABET M}}_{t}\}_{t=1}^T$ are the optimal value functions for MDP~${\ALPHABET M}$.
\end{theorem}
We will state and prove a more general version of this result in the next subsection. 

\begin{remark}
Certainty equivalent policies are not appropriate for all models. For instance, in some POMDPs, before taking a control action, the agent has the option to pay a cost to take a sensing action that reveals the true state of the MDP. In such models, the certainty equivalent policy will never choose the sensing action. Therefore, if the sensing action is not too costly, a policy that occasionally pays the sensing cost to learn the true state may outperform certainty equivalent policy. 
\end{remark}

\subsection{Certainty equivalent policies using State Abstraction}\label{sec:abstract}

In problems with large or continuous state spaces, it can be difficult to find an optimal policy $\pi^{\ALPHABET M}$ of MDP~$\ALPHABET M$ due to the  curse of dimensionality. For such large-scale MDPs, one typically obtains an approximately optimal policy for MDP~$\ALPHABET M$ by considering an abstract MDP obtained by state aggregation or state quantization. In such situations, it is natural to consider a certainty equivalent policy based on an optimal policy of the abstract MDP and estimates of the abstract state. In this section, we formally define such a policy and present a generalization of Theorem~\ref{thm:main_theorem} to such settings.

Suppose there is an abstract state space $\ABSTRACT {\ALPHABET S}$ which is equipped with a metric $d_{\ABSTRACT {\ALPHABET S}}$, a (measurable) state abstraction function $\phi \colon \ALPHABET S \to \ABSTRACT {\ALPHABET S}$, and two stochastic kernels $\lambda^P, \lambda^c \colon \ABSTRACT {\ALPHABET S} \to \Delta(\ALPHABET S)$ such that for each $\ABSTRACT s_t \in \ABSTRACT {\ALPHABET S}$, $\lambda^P(\phi^{-1}(\ABSTRACT s_t)|\ABSTRACT s_t) = 1$ and $\lambda^c(\phi^{-1}(\ABSTRACT s_t) | \ABSTRACT s_t) = 1$.

We construct an abstract MDP $\ABSTRACT {\ALPHABET M} \coloneqq \langle \ABSTRACT{\ALPHABET S}, \ALPHABET A$, $\{\ABSTRACT P_t\}_{t=1}^{T-1}$, $\{\ABSTRACT c_t\}_{t=1}^T, T \rangle$ where the dynamics $\ABSTRACT P_t \colon \ABSTRACT {\ALPHABET S} \times \ALPHABET A \to \ABSTRACT {\ALPHABET S}$ and the per-step cost $\ABSTRACT c_t \colon \ABSTRACT {\ALPHABET S} \times \ALPHABET A \to \reals$ are defined as follows: for any measurable $M_{\ABSTRACT {\ALPHABET S}} \subset \ABSTRACT {\ALPHABET S}$, 
\begin{align}
    \ABSTRACT P_t(\ABSTRACT S_{t+1}\in & M_{\ABSTRACT {\ALPHABET S}}|\ABSTRACT s_t,a_t)  \notag \\
    &= \int_{\phi^{-1}(\ABSTRACT s_t)} P_{\ALPHABET S,t}\bigl(\phi(S_{t+1})\in  M_{\ABSTRACT {\ALPHABET S}}|s_t,a_t)\bigr)\lambda^P(ds_t|\ABSTRACT s_t)
    \label{eq:abstract-P}
\end{align}
and
\begin{equation}\label{eq:abstract-c}
    \ABSTRACT c_t(\ABSTRACT s_t,a_t) = \int_{\phi^{-1}(\ABSTRACT s_t)}c_t(s_t,a_t)\lambda^c(ds_t|\ABSTRACT s_t).
\end{equation}

The cost function of the abstract MDP can be viewed as a weighted averaging  of the original MDP cost over all states in $\phi^{-1}(\ABSTRACT s_t)$; a similar interpretation applies for the dynamics in the abstract model as well. 

Note that when $\ABSTRACT {\ALPHABET S} = \ALPHABET S$ and $\phi(s) = s$, the abstract MDP $\ABSTRACT{\ALPHABET M}$ is equal to the MDP $\ALPHABET M = \langle \ALPHABET S, \ALPHABET A, \{P_{S,t}\}_{t=1}^{T-1}, \{c_t\}_{t=1}^T, T \rangle$.

We impose the measurable selection assumption on  MDP~$\ABSTRACT {\ALPHABET M}$.
\begin{assumption}\label{ass:meas-selection2}
    The model $\ABSTRACT {\ALPHABET M}$ satisfies measurable selection.
\end{assumption}

As in Section \ref{sec:gce}, an implication of measurable selection is that there exists an optimal policy $\pi^{\ABSTRACT {\ALPHABET M}} = (\pi^{\ABSTRACT {\ALPHABET M}}_1, \dots, \pi^{\ABSTRACT {\ALPHABET M}}_T)$, where $\pi^{\ABSTRACT {\ALPHABET M}}_t \colon \ABSTRACT {\ALPHABET S}  \to \ALPHABET A$, with associated optimal value functions $(V^{\ABSTRACT {\ALPHABET M}}_1, \dots, V^{\ABSTRACT {\ALPHABET M}}_T)$, $V^{\ABSTRACT {\ALPHABET M}}_t \colon \ABSTRACT {\ALPHABET S} \to \reals$, for MDP $\ABSTRACT{\ALPHABET M}$~\cite{hernandez2012discrete}. 

We now use the optimal policy $\pi^{\ABSTRACT {\ALPHABET M}}$ for the MDP~$\ABSTRACT {\ALPHABET M}$ to define a feasible policy for the POMDP~$\ALPHABET P$. This policy is similar to the certainty equivalent policies of Section \ref{sec:gce} except that (i) we estimate the abstract state $\ABSTRACT S_{t}$, and (ii) use an optimal policy of $\ABSTRACT {\ALPHABET M}$. For convenience, we reuse some of the notation of Section \ref{sec:gce}. 

Suppose we are given a sequence of \emph{abstract state estimation functions} $\{\EST_t\}_{t=1}^T$, where $\EST_t \colon \ALPHABET H_t \to \ABSTRACT {\ALPHABET S}$. 
We say that a history-dependent policy $\mu^\EST = (\mu^\EST_1, \dots, \mu^\EST_T)$ is \emph{ certainty equivalent} with respect to the abstract MDP $\ABSTRACT {\ALPHABET M} =\langle \ABSTRACT{\ALPHABET S}, \ALPHABET A, \{\tilde {P}_{t}\}_{t=1}^{T-1}, \{\tilde {c}_t\}_{t=1}^T, T\rangle$ and estimators  $\{\EST_t\colon \ALPHABET H_t \to \ABSTRACT {\ALPHABET S}\}_{t \ge 1}$ if 
\begin{align}\label{eq:cert_equiv2}
    \mu^\EST_t(h_t) = \pi^{\ABSTRACT {\ALPHABET M}}_t(\EST_t(h_t)).
\end{align}
In other words, a certainty equivalent policy treats   $\EST_t(h_t)$ as an error-free estimate of the state of the abstract MDP and uses the optimal policy of $\ABSTRACT {\ALPHABET M}$ to take its action.  As before, we are interested in an upper bound on the gap between the value functions of policy $\mu^\EST$ and the optimal value functions of the POMDP, i.e. a bound on  
 $W^{{\ALPHABET{P}},\mu^\EST}_t(h_t) - W^{\ALPHABET{P}}_t(h_t)$.

We impose the following technical assumption on the model and the state abstraction.
\begin{assumption}\label{ass:lipschitz2} 
    There exist a sequence of concave and non-decreasing functions $F^P_t, F^c_t \colon \reals_{\ge 0} \to \reals_{\ge 0}$, $t \in \{1, \dots, T\}$, such that for any $s, s'\in \ALPHABET S$ and $a \in \ALPHABET A$, we have
    \begin{equation}\label{eq:Lip-P-2}
        \dWas(P^{\phi}_{S,t}(\cdot | s, a), P^{\phi}_{S,t}(\cdot | s', a)) \le F^P_t(\dhSphi(s, s')),
    \end{equation}
    and
    \begin{equation}\label{eq:Lip-c-2}
        \bigl| c_t(s,a) - c_t(s',a) \bigr| \le F^c_t(\dhSphi(s, s'))
    \end{equation}
    where $P^{\phi}_{S,t}$ is a stochastic kernel from $\ALPHABET S \times \ALPHABET A$ to $\Delta(\ABSTRACT {\ALPHABET S})$ defined as
    \[
        P^{\phi}_{S,t}(M_{\ABSTRACT {\ALPHABET S}} | s_t, a_t) = P_{\ALPHABET S,t}\bigl(\phi(S_{t+1})\in  M_{\ABSTRACT {\ALPHABET S}}|s_t, a_t\bigr)
    \]
    for all Borel subsets $M_{\ABSTRACT {\ALPHABET S}}$ of $\ABSTRACT {\ALPHABET S}$.
\end{assumption}
Assumption~\ref{ass:lipschitz2} implies that $\phi: \ALPHABET S \to \ABSTRACT{\ALPHABET S}$ is a good state abstraction in the following sense: If $\phi(s)$ is close to $\phi(s')$, then for any action $a$, the per-step costs at $s$ and $s'$ are close and the probability distributions of the next abstracted state given $s$ and $s'$ are close.

Our results depend on the {quality} of the estimates produced by  $\EST_t$. For that purpose, for each time~$t$ we define the worst-case value of the conditional expected estimation error
\begin{align}\label{eq:closeness_assm2}
      \eta_t &:= \sup_{h_t \in \ALPHABET H_t} \EXP[ \dhS(\phi(S_t), \EST_t(h_t)) |  h_t ].
    \end{align}
As we stated in Remark~\ref{rem:policy_indep}, the right hand side in Eq.~\eqref{eq:closeness_assm2} does not depend on the policy because conditional probability distribution of the state $S_t$ given history $h_t$ is policy independent~\cite{Astrom1965,Smallwood1973}. 
\begin{assumption}\label{ass:finite_eta2}
For $t=1,2,\ldots,T,$ we have $\eta_t < \infty$ where $\eta_t$ is given by~\eqref{eq:closeness_assm2}.
\end{assumption}

We can now state the main result.
\begin{theorem}\label{thm:main_theorem2}
Define
\[\varepsilon_t = F^c_t (\eta_t) \text{ and } \delta_t = F^P_t (\eta_t) + \eta_{t+1}\]
where $\eta_t$ is given by \eqref{eq:closeness_assm2}. Then, under Assumptions~\ref{ass:meas-selection2}, \ref{ass:lipschitz2} and \ref{ass:finite_eta2}, we have that the certainty equivalent policy $\mu^\EST$ (defined in~\eqref{eq:cert_equiv2}) satisfies 
    \begin{equation}\label{eq:main_theorem_rep_ex2}
        W^{\ALPHABET P, \mu^\EST}_t(h_t) - W^{\ALPHABET P}_t(h_t) \le 2 \alpha_t
    \end{equation}
    where
    \begin{equation}\label{eq:alpha}
        \alpha_t = \varepsilon_t + \sum_{\tau = t}^{T-1}\bigl[ \delta_{\tau} \Lip(V^{\ABSTRACT{\ALPHABET M}}_{\tau+1}) + \varepsilon_{\tau+1} \bigr]
    \end{equation}
    and $\{V^{\ABSTRACT{\ALPHABET M}}_{t}\}_{t=1}^T$ are the optimal value functions for MDP~$\ABSTRACT{\ALPHABET M}$.
\end{theorem}
We will prove this result in Section~\ref{sec:analysis}. Note that when $\ABSTRACT {\ALPHABET S} = \ALPHABET S$ and $\phi(s) = s$, Theorem~\ref{thm:main_theorem2} reduces to Theorem~\ref{thm:main_theorem}. As illustrated in the corollary below, Theorem~\ref{thm:main_theorem2} also provide sub-optimality bounds for using a policy of an abstract MDP in the original MDP, which may be viewed as a finite horizon version of~\cite{gelada2019deepmdp}.

\begin{corollary}\label{cor:eta_zero}
    For any Markovian policy $\pi$ of MDP~$\ALPHABET M$, let $V^{\ALPHABET M, \pi}_t \colon \ALPHABET S \to \reals$ denote the value function of policy~$\pi$. Given the optimal policy $\pi^{\tilde {\ALPHABET M}} = (\pi^{\tilde{\ALPHABET M}}_1, \dots, \pi^{\tilde {\ALPHABET M}}_T)$ of the abstract MDP~$\tilde {\ALPHABET M}$, define a feasible policy $\bar \pi$ of MDP~$\ALPHABET M$ as follows: $\bar \pi = (\pi^{\tilde {\ALPHABET M}}_1 \circ \phi, \dots$, ${\pi^{\tilde {\ALPHABET M}}_T \circ \phi})$. Then, under Assumptions~\ref{ass:meas-selection2} and~\ref{ass:lipschitz2}, for any time~$t$ and any realization $s_t$ of $S_t$, we have
    \[
        V^{\ALPHABET M, \bar \pi}_t(s_t) - V^{\ALPHABET M}(s_t) \le 2 \alpha_t
    \]
    where $\alpha_t$ is given by~\eqref{eq:alpha} with $\varepsilon_t = F^c_t(0)$ and $\delta_t = F^P_t(0)$. 
\end{corollary}
\begin{proof}
    This is an immediate consequence of Theorem~\ref{thm:main_theorem2} by considering the trivial setting where $Y_t = S_t$ (and thus, $h_t = (s_{1:t}, a_{1:t-1})$) and take $\EST(h_t) = \phi(s_t)$, which implies $\eta_t = 0$ for all $t$. In this case, $W^{\ALPHABET P}_t(h_t) = V^{\ALPHABET M}_t(s_t), \mu^\EST = \bar \pi$ and $W^{\ALPHABET P,\mu^\EST}_t(h_t) = V^{\ALPHABET M,\bar \pi}_t(s_t)$. 
\end{proof}

\section{Illustrative examples} \label{sec:examples}
In this section we present several examples to illustrate observation models (and corresponding estimators) where certainty equivalent policies may be useful. We apply our results to derive explicit bounds on the sub-optimality of certainty equivalent policies for specific observation models. 

\subsection{Bounded observation noise}\label{sec:bounded}

\subsubsection{System model}
Consider a POMDP where $\ALPHABET Y = \ALPHABET S$ and the system dynamics $P_t$ are such that
\[
    d_{\ALPHABET S}(Y_t, S_t) \le r,
\]
where $r \in [0, \infty)$. Moreover, we assume that the MDP model $\ALPHABET M$ satisfies measurable selection (Assumption~\ref{ass:meas-selection}) and  that the dynamics and per-step cost are Lipschitz, i.e., there exist non-negative finite constants $L_t^P$ and $L_t^c$ such that Assumption~\ref{ass:lipschitz} is satisfied with $F_t^P(x)=L_t^P x$ and $F_t^c (x) = L_t^c x$ (see Remark~\ref{rem:lip_linear}).

\subsubsection{Certainty equivalent policy}
For this example, we consider certainty equivalent policies with respect to the original MDP $\ALPHABET M$, i.e., take the state abstraction function $\phi(s) = s$.  Furthermore, we take the state estimate to be the last observation, i.e., $\EST_t(h_t) = y_t$. Then, the certainty equivalent policy is given by
\[
    \mu_t^{\EST}(h_t) = \pi^{\ALPHABET M}_t(y_t).
\]

\subsubsection{Sub-optimality bound}
We have assumed that Assumptions~\ref{ass:meas-selection} and~\ref{ass:lipschitz} are satisfied. Moreover,
\[
    \EXP[ \dS(S_t, \EST_t(H_t)) | h_t ] = \EXP\bigl[d_{\ALPHABET S}(S_t,Y_t)| h_t\bigr] 
    \le r.
\]
Therefore, $\eta_t \le r$ and Assumption~\ref{ass:finite_eta} is also satisfied. Furthermore, the $\varepsilon_t$ and $\delta_t$ in Theorem~\ref{thm:main_theorem} can be upper bounded by
\[
    \varepsilon_t \le r L_t^c 
    \quad\text{and}\quad
    \delta_t \le r(1 + L_t^P).
\]
Hence, the bound in Theorem~\ref{thm:main_theorem} can be explicitly written as 
\begin{equation}
    W^{\ALPHABET P,\mu^\EST}_t(h_t) - W^{\ALPHABET P}_t(h_t) \le 2r L_t^{\ALPHABET M}
\end{equation}
where
\begin{equation}\label{eq:LT}
    L^{\ALPHABET M}_t = \biggl[L_t^c + \sum_{\tau = t}^{T-1}\bigl[ (1 + L_{\tau}^P) \Lip(V^{\ALPHABET M}_{\tau+1}) + L_{\tau+1}^c \bigr] \biggr]
\end{equation}
is a constant that depends on the Lipschitz constants of the dynamics, per-step cost, and the optimal MDP value function.

This bound scales linearly with $r$, which means that as the observation becomes closer to the underlying state (i.e., ``observation noise'' becomes small), the performance of the certainty equivalent policy approaches that of the optimal POMDP policy. 

\subsection{Intermittently degraded observation}

\subsubsection{System model}
Consider a POMDP where $\ALPHABET Y = \ALPHABET S$ and the system dynamics $P_t$ is such that the controller either gets a good observation (indicated by event $E_t$) or a bad observation (indicated by $E_t^c$). Under $E_t$, 
\(
    d_{\ALPHABET S}(Y_t, S_t) \le r
\)
while under $E_t^c$, 
\(
    d_{\ALPHABET S}(Y_t, S_t) \le R
\)
where $0 \le r \le R < \infty$.  We assume that for any history $h_t$, $\PR(E_t^c | h_t) \le p$. 

Moreover, we assume that the MDP model $\ALPHABET M$ satisfies measurable selection (Assumption~\ref{ass:meas-selection}) and  that the dynamics and per-step cost are Lipschitz, i.e., there exist non-negative finite constants $L_t^P$ and $L_t^c$ such that Assumption~\ref{ass:lipschitz} is satisfied with $F_t^P(x)=L_t^P x$ and $F_t^c (x) = L_t^c x$ (see Remark~\ref{rem:lip_linear}).

\subsubsection{Certainty equivalent policy}
As for the example in Sec.~\ref{sec:bounded}, we take $\phi(s) = s$ and $\EST_t(h_t) = y_t$. Then, the certainty equivalent policy is given by
\[
    \mu_t^{\EST}(h_t) = \pi^{\ALPHABET M}_t(y_t).
\]

\subsubsection{Sub-optimality bound}
We have assumed that Assumption~\ref{ass:meas-selection} and~\ref{ass:lipschitz} are satisfied. Moreover,
\begin{align}
    \EXP[ \dS(S_t, \EST_t(H_t)) | h_t ]  &= \EXP\bigl[d_{\ALPHABET{S}}(S_t,Y_t)| h_t\bigr] \nonumber \\
    &\le (1-p) r + p R
\end{align}
Hence, we have $\eta_t \leq (1-p)r + pR$ and Assumption~\ref{ass:finite_eta} is also satisfied. Furthermore, the $\varepsilon_t$ and $\delta_t$ in Theorem~\ref{thm:main_theorem} can be upper bounded by
\[
    \varepsilon_t \le [ (1-p)r + p R ] L_t^c \quad\text{and}\quad \delta_t \le [(1-p) r + p R ] (1 + L_t^P).
\]
Therefore, the bound in Theorem~\ref{thm:main_theorem} can be explicitly written as 
\begin{align}
    &W^{\ALPHABET P,\mu^\EST}_t(h_t) - W^{\ALPHABET P}_t(h_t)
    \leq 2[(1-p) r + p R ] L_t^{\ALPHABET M}
\end{align}
where $L^{\ALPHABET M}_t$ is as defined in~\eqref{eq:LT}.

These results demonstrate that when the state estimation error $[(1-p) r + p R ]$ is small, either due to observations with small noise or observations that are frequently accurate, certainty equivalent policies can perform near-optimally. The bounds provide a quantitative measure of the sub-optimality in terms of the estimation error and the Lipschitz constants of the model.

\subsection{Bounded observation noise with state quantization}

\subsubsection{System model}
Consider the bounded observation noise model of Sec.~\ref{sec:bounded} where $\ALPHABET Y = \ALPHABET S$ and the observation model is such that $d_{\ALPHABET S}(Y_t, S_t) \le r$. In addition, an abstract model $\hat {\ALPHABET M}$ is given, which is constructed by partitioning the state space $\ALPHABET S$ into a finite collection $\{\Psi_k\}_{k=1}^{K}$ of quantization cells, each with a representative element $\ABSTRACT s_k \in \Psi_k$. The abstract state space is $\ABSTRACT {\ALPHABET S} = \{\ABSTRACT s_1, \dots, \ABSTRACT s_K\}$ and the state abstraction function is 
$\phi: \ALPHABET S \to \ABSTRACT{\ALPHABET S}$ given by $\phi(s) = \ABSTRACT s_k$ for all $s \in \Psi_k$, $k \in \{1, \dots, K\}$. 
Moreover, the stochastic kernels $\lambda^P(\cdot \mid \ABSTRACT s_k)$ and $\lambda^c(\cdot \mid \ABSTRACT s_k)$ are Dirac delta measures on~$\ABSTRACT {s}_k$, $k \in \{1,\dots, K\}$. The abstract dynamics $\tilde P$ and abstract cost $\tilde c$ are constructed as in~\eqref{eq:abstract-P} and~\eqref{eq:abstract-c}. We also define a metric $d_{\ABSTRACT{\ALPHABET S}}$ on $\ABSTRACT{\ALPHABET S}$ by
\[
    d_{\ABSTRACT{\ALPHABET S}}(\ABSTRACT s_1, \ABSTRACT s_2) \coloneqq d_{\ALPHABET S}(\ABSTRACT s_1, \ABSTRACT s_2), \quad \forall \ABSTRACT s_1, \ABSTRACT s_2 \in \ABSTRACT{\ALPHABET S} \subset \ALPHABET S,
\]

We assume Assumptions~\ref{ass:meas-selection2} and~\ref{ass:lipschitz2} are satisfied.

\subsubsection{Certainty equivalent policy}
For this example, we consider certainty equivalent policies with respect to the abstract MDP $\ABSTRACT{\ALPHABET M}$.  Furthermore, 
we take the abstract state estimate to be the quantized value of the last observation, i.e., $\EST_t(h_t) = \phi(y_t)$. Then, the certainty equivalent policy is 
\[
    \mu_t^{\EST}(h_t) = \pi^{\ABSTRACT M}\bigl(\phi(y_t)\bigr).
\]

\subsubsection{Sub-optimality bound}
We have assumed that Assumption~\ref{ass:meas-selection2} and~\ref{ass:lipschitz2} are satisfied. Let
$R$ denote the maximum radius of a quantization cell, i.e., 
\[
   R \coloneqq \max_{k \in \{1, \dots, K\}} \sup_{s \in \Psi_k}d_{\ALPHABET S}(s, \phi(s)).
\]
Then, by triangle inequality, we have
\begin{align}
  &\EXP[ \dhS(\phi(S_t), \EST_t(h_t)) |  h_t ]  
  = \EXP[ \dS(\phi(S_t), \phi(Y_t) |  h_t ] \nonumber \\
   &\leq \EXP \bigl[ \dS\bigl(\phi(S_t),S_t) + \dS (S_t, Y_t) + \dS\bigl(Y_t, \phi(Y_t)\bigr) | h_t \bigr] \nonumber \\
   &\leq R + r + R = r + 2R.
\end{align}
Hence, we have $\eta_t \leq r + 2R$ and Assumption~\ref{ass:finite_eta2} is also satisfied. Furthermore, the $\varepsilon_t$ and $\delta_t$ in Theorem~\ref{thm:main_theorem2} can be upper bounded by
\[
    \varepsilon_t \leq F_t^c(r + 2R)
    \quad\text{and}\quad
    \delta_t \leq F_t^P(r + 2R) + r + 2R.
\]
Therefore, we can upper bound the sub-optimality of using the certainty equivalent policy by the bound~\eqref{eq:main_theorem_rep_ex2} in Theorem~\ref{thm:main_theorem2}.

\subsection{Certainty equivalence in learning/adaptive control}
\subsubsection{System model}
Consider a parameterized MDP $\ALPHABET M_X(\theta)$ with state space $\ALPHABET X$, action space $\ALPHABET A$, time-invariant dynamics $P_{X,\theta}$ and time-invariant per-step cost $\ell_{\theta}$, where the parameters $\theta \in \Theta$ and are distributed according to some probability distribution~$P_{\Theta}$ independent of noise in the dynamics. We assume that $\ALPHABET X$ and $\Theta$ are metric spaces with metrics $d_{\ALPHABET X}$ and $d_{\Theta}$, respectively.

The controller doesn't know the parameters $\theta$ but knows the history of state and actions $h_t = (x_{1:t}, a_{1:t-1})$. The above model can be viewed as an POMDP with state space $\ALPHABET S = \ALPHABET X \times \Theta$, observation space $\ALPHABET X \times \reals$, where $S_t = (X_t, \theta)$ and $Y_t = (X_t, \ell(X_{t-1}, A_{t-1}))$. We take $d_{\ALPHABET S}((x_1,\theta_1), (x_2,\theta_2)) = d_{\ALPHABET X}(x_1, x_2) + d_{\Theta}(\theta_1, \theta_2)$.  Observe that the MDP $\ALPHABET M$ corresponding to this POMDP is equivalent to $\ALPHABET M_X(\theta)$. We denote the optimal policy of this MDP by $\pi^{\ALPHABET M}( x, \theta) = \pi^{\ALPHABET M_X(\theta)}(x)$.

We assume that for all $\theta \in \Theta$, the model $\ALPHABET M_X(\theta)$ satisfies measurable selection (Assumption~\ref{ass:meas-selection}). Moreover,
there exist non-negative finite constants $L^P$ and $L^c$ such that for any $x, x' \in \ALPHABET X$ and $\theta, \theta' \in \Theta$, we have
\begin{align}\label{eq:Lip-P-learn}
    \dWas(P_{X, \theta}(\cdot | x, a),& P_{X, \theta'}(\cdot | x', a)) \notag \\
    &\le L^P (d_{\ALPHABET X}(x,x') + d_{\Theta}(\theta,\theta')),
\end{align}
and
\begin{equation}\label{eq:Lip-c-learn}
    \bigl| \ell_{\theta}(x,a) - \ell_{\theta'}(x',a) \bigr| \leq L^c (d_{{\ALPHABET X}}(x,x') + d_{\Theta}(\theta,\theta')).
\end{equation}

\subsubsection{Certainty equivalent policy}

As in the previous examples, we consider $\phi(s) = s$ but consider a general estimator $\EST_t(h_t) = (x_t, \hat \theta_t)$ where $\hat\theta_t$ is some estimate of $\theta$ based on $h_t$, e.g.,  the MMSE estimator $\ESTIMATE\theta_t = \EXP[\theta | h_t]$. Then, the certainty equivalent policy is
\[
    \mu^{\EST}_t(h_t) = \pi_t^{\ALPHABET M}(x_t, \hat\theta_t).
\]

\subsubsection{Sub-optimality bound}
We have assumed that Assumption~\ref{ass:meas-selection} and~\ref{ass:lipschitz} are satisfied. Moreover,
\begin{align}
\eta_t &= 
    \sup_{h_t \in \ALPHABET H_t} \EXP[ \dS(S_t, \EST_t(H_t)) | h_t ] 
    =\sup_{h_t \in \ALPHABET H_t} \EXP[ d_{\Theta}(\theta_t, \hat \theta_t) | h_t].
    \label{eq:eta_mdp_learning}
\end{align}
Thus, if Assumption~\ref{ass:finite_eta} holds, the $\varepsilon_t$ and $\delta_t$ in Theorem~\ref{thm:main_theorem} can be upper bounded by
\[
    \varepsilon_t \le  L^c \eta_t \quad\text{and}\quad \delta_t \le L^P \eta_t + \eta_{t+1}.
\]
Therefore, we can bound the sub-optimality in using the certainty equivalent policy by the bound~\eqref{eq:main_theorem_rep_ex1} given in Theorem~\ref{thm:main_theorem}. 

These results show that the performance of certainty equivalent policies depends on the performance $\eta_t$ of the parameter estimation. If $\eta_t$ decays sufficiently fast, e.g., exponentially fast, then we can obtain uniform upper bounds on the performance error $2 \alpha_t$ in~\eqref{eq:main_theorem_rep_ex1} even when $T \to \infty$.

\subsection{Control with event-triggered communication}

\subsubsection{System model}
Consider a system consisting of a plant, a sensor co-located with the plant, and a remote controller. Let $X_t \in \ALPHABET X$ and $A_t \in \ALPHABET A$ denote the state and control input of the plant. The state of the plant evolves according to a controlled transition kernel $P_{X,t} \colon \ALPHABET X \times \ALPHABET A \to \Delta(\ALPHABET X)$. 

The sensor observes the current state $X_t$ and decides whether or not to transmit the state. Let $Y_t \in \ALPHABET X \cup \{ \mathfrak{E} \}$ denote the observation of remote controller, i.e.,
\[
    Y_t = \begin{cases}
        \mathfrak{E} & \text{if sensor does not communicate}\\
        X_t       & \text{if sensor communicates}
    \end{cases}
\]
where $\mathfrak{E}$ denotes a null observation.

The remote controller generates the action $A_t$ according to a general history dependent policy $\mu = (\mu_1, \dots, \mu_T)$ and incurs a per-step cost $c_t(x_t,a_t)$. 

The above problem is a decentralized control problem where both the communication policy at the sensor and the control policy at the remote controller have to be determined. We consider the setting when the communication policy is fixed to be an \emph{event-triggered} communication policy~\cite{molin2014suboptimal, mazo2011decentralized,heemels2012introduction}, which operates as follows. It is assumed that 
the remote controller keeps track of a state estimate $\hat X_{t|t} \in \ALPHABET X$ as follows:
\begin{equation}\label{eq:X-t-t}
    \hat X_{t|t} = \begin{cases}
        Y_t & \text{if } Y_t \neq \mathfrak{E} \\
        \hat X_{t|t-1}, & \text{if } Y_t = \mathfrak{E}
    \end{cases}
\end{equation}
where $\hat X_{1|0} = \EXP[X_1]$ and 
\begin{equation}\label{eq:X-t-t-1}
    \hat X_{t|t-1} = g(\hat X_{t-1|t-1}, A_{t-1}), \quad t > 1,
\end{equation}
where $g \colon \ALPHABET X \times \ALPHABET A \to \ALPHABET X$ is a pre-specified update function.

In an event-triggered policy, the sensor transmits when the following inequality is satisfied
\[
    d_{\ALPHABET X}(X_t, \hat X_{t|t-1}) > r.
\]
where $r$ is a pre-specified constant. This policy ensure that $d_{\ALPHABET X}(X_t, \hat X_{t|t}) \le r$. The objective then is to find the best control strategy at the remote controller.

Once the event triggered policy is fixed, the above model corresponds to a POMDP $\ALPHABET P$ with state $S_t = (X_t, \hat X_{t|t-1})$, observation $Y_t$, and action $A_t$. Let $\ALPHABET M$ denote the corresponding MDP, in which the controller has access to $S_t$. 

We consider an abstract MDP $\ABSTRACT{\ALPHABET M}$ with state space $\ALPHABET X$ which is constructed using a state abstraction function $\phi(x, \hat x) = x$ and stochastic kernels $\lambda^P(\cdot | x)$ and $\lambda^c(\cdot | x)$ as Dirac delta measures on $(x,x)$. Then, the dynamics $\tilde P_t$ of the abstract MDP is equal to $P_{X,t}$ and the per-step cost $\tilde c_t$ is equal to $c_t$. Thus, MDP $\tilde M = \langle \ALPHABET X, \ALPHABET A, \{ P_{X,t}\}_{t=1}^{T-1}, \{ c_t\}_{t=1}^T \rangle$.  We assume that $\ABSTRACT{\ALPHABET M}$ satisfies Assumption~\ref{ass:meas-selection2} and there exist concave and non-decreasing functions $F^P_t, F^c_t \colon \reals_{\ge 0} \to \reals_{\ge 0}$, $t \in \{1, \dots, T\}$, such that for any $x,x' \in \ALPHABET X$ and $a \in \ALPHABET A$, we have
\begin{equation}\label{eq:dynamics-event-triggered-lip}
     \dWas(P_{X,t}(\cdot | x, a), P_{X,t}(\cdot | x', a)) \le F^P_t(d_{\ALPHABET X}(x, x'))
\end{equation}
and
\begin{equation}\label{eq:cost-event-triggered-lip}
    \bigl| c_t(x,a) - c_t(x',a) \bigr| \le F^c_t(d_{\ALPHABET X}(x, x')).
\end{equation}
Assumption~\ref{ass:meas-selection2} implies that there exists an optimal policy for MDP $\ABSTRACT{\ALPHABET M}$, which we denote by $\pi^{\ABSTRACT {\ALPHABET M}}$. Moreover, it can be verified that~\eqref{eq:dynamics-event-triggered-lip} and~\eqref{eq:cost-event-triggered-lip} implies Assumption~\ref{ass:lipschitz2}.

\subsubsection{Certainty equivalent policy}
For this example, we consider certainty equivalent policies with respect to the abstract MDP $\ABSTRACT{\ALPHABET M}$.  Furthermore, we take the state estimate to be $\EST_t(h_t) = \hat x_{t|t}$, which is recursively computed from the history using~\eqref{eq:X-t-t}. Then, the certainty equivalent policy is given by
\[
    \mu_t^{\EST}(h_t) = \pi^{\ABSTRACT{\ALPHABET M}}_t(\hat x_{t|t}).
\]

\subsubsection{Sub-optimality bound}
We have assumed Assumption~\ref{ass:meas-selection2} and assumed sufficient conditions that imply Assumption~\ref{ass:lipschitz2}. 
Recall that the event-triggered sensor transmission policy ensures that $d_{\ALPHABET X}(X_t, \hat X_{t|t}) \le r$.
Therefore,
\[
    \EXP[d_{\ALPHABET X}(\phi(S_t),\EST(H_t))|h_t] = \EXP[d_{\ALPHABET X}(X_t, \hat X_{t|t}) | h_t] \leq r.
\]
Hence, $\eta_t \leq r$ and Assumption~\ref{ass:finite_eta2} is satisfied.  Furthermore, the $\varepsilon_t$ and $\delta_t$ in Theorem~\ref{thm:main_theorem2} can be upper bounded by
\[
    \varepsilon_t \leq F_t^c(r)
    \quad\text{and}\quad
    \delta_t \leq F_t^P(r) + r. 
\]
Therefore, we can upper bound the sub-optimality of using the certainty equivalent policy by the bound~\eqref{eq:main_theorem_rep_ex2} in Theorem~\ref{thm:main_theorem2}. These bounds quantify the sub-optimality gap of certainty equivalent control with event-triggered sensing.

\subsection{Control of non-homogeneous multi-particle systems}

\subsubsection{System model}
Consider a system consisting of $n$ particles, where each particle $i$, $i \in \{1, \dots, n\}$, has a state $X^i_t \in \reals$. The global state of the system is $X_t = (X^1_t, X^2_t, \ldots, X^n_t)^\TRANS \in \reals^n$. The observation is given by 
\[
    Y_t = X_t + N_t
\]
where $N_t = (N^1_t, N^2_t, \ldots, N^n_t)^\TRANS $ is the observation noise with $N^i_t \in [-r^i,r^i]$ for some $r^i \ge 0$.

Let $M_t = \sum_{i=1}^n \alpha^i X^i_t$ denote the weighted mean of the global state, where $\alpha^i$, $i \in \{1, \dots, n\}$, are non-negative weights that add to $1$.  We assume that the dynamics of each particle is given by
\begin{equation}\label{eq:mean-field-dynamics}
    X_{t+1}^i = \bar f(M_t,A_t, W_t) + f^i(X_t,A_t,W_t)
\end{equation}
where $A_t \in \ALPHABET A$ is the control action at time~$t$ and $\{W_t\}_{t \ge 1}$, $W_t \in \ALPHABET W$, is an independent and identically distributed process with distribution $P_W$. 

\begin{assumption}\label{assm:mean-field-dynamics}
We assume the following.
\begin{enumerate}
\item The function $\bar f$ is $L^{\bar f}$-Lipschitz in the first argument, i.e., for all $m,m' \in \reals$, $a \in \ALPHABET A$, and $w \in \ALPHABET W$,
\begin{equation}
    |\bar f(m,a,w) - \bar f(m',a,w)| \leq L^{\bar f}|m - m'|
\end{equation}
\item  There exists constants $\gamma^i$ such that 
$\| f^i \|_{\infty} \leq \gamma^i$ for $i \in \{1, \dots, n\}$.
\end{enumerate}
\end{assumption}
The per-step cost is given by
\begin{equation}\label{eq:mean-field-cost}
    c(X_t,A_t) = \bar \ell(M_t, A_t) + \ell(X_t, A_t).
\end{equation}
\begin{assumption}\label{assm:mean-field-cost}
We assume the following. 
\begin{enumerate}
   \item The function $\bar \ell$ is $L^{\bar \ell}$-Lipschitz in the first argument, i.e. for all $m, m'\in \reals$ and $a \in \ALPHABET A$, 
   \begin{equation}
       |\bar \ell(m,a) - \bar \ell(m', a)| \leq L^{\bar \ell} |m - m'|.
   \end{equation}
   \item There exists a constant $\beta$ such that 
   $\|\ell\|_{\infty} \leq \beta$.
\end{enumerate}
\end{assumption}
The above model is a POMDP $\ALPHABET P$ with state $S_t = X_t$, observation $Y_t$, action $A_t$, and per-step cost $c(X_t,A_t)$.

\subsubsection{Certainty equivalent policy}
We consider an abstract MDP that focuses on the weighted mean of the global state. The abstract state space is $\tilde{\ALPHABET S} = \reals$ and the state abstraction function $\phi: \reals^n \to \reals$ is given by
\[
    \phi((x^1,x^2,\ldots,x^n)^\TRANS) = \sum_{i=1}^n \alpha^ix^i
\]

with
\[
    \phi^{-1}(m) = \Bigl\{(x^1,x^2,\ldots,x^n)^\TRANS: \sum_{i=1}^n \alpha^ix^i = m \Bigr\}.
\]

We assume that the metric on the abstract state space is $d_{\ABSTRACT{\ALPHABET S}}(m, m') = |m - m'|$ and take $\lambda^P(\cdot|m)$ and $\lambda^c(\cdot|m)$ to be delta distributions at $m \textbf{1}_n$, where $\mathbf{1}_n$ is the n-dimensional vector of ones.

Therefore, the abstract per-step cost is given by
\begin{align*}
    \ABSTRACT c(\ABSTRACT{s}, a) &= c(\ABSTRACT{s} \mathbf{1}_n,a) 
    = \bar \ell(\ABSTRACT{s},a) + \ell (\ABSTRACT{s} \mathbf{1}_n),
\end{align*}
and the abstract state dynamics are given by
\begin{align*}
    \ABSTRACT{S}_{t+1} = \bar f(\ABSTRACT S_t,A_t,W_t) + \sum_{i=1}^n \alpha^i f^i(\ABSTRACT S_t \mathbf{1}_n,A_t,W_t).
\end{align*}
The above model corresponds to an abstract MDP $\tilde{\ALPHABET M}$. We assume that $\tilde{\ALPHABET M}$ satisfies Assumption~\ref{ass:meas-selection2}.

We consider the weighted mean of the last observation as an estimate of the abstract state, i.e., 
\[
   \EST_t(h_t) = \sum_{i=1}^n \alpha^i y^i_t.
\]
Then, the certainty equivalent policy is
\[
    \mu_t^{\EST}(h_t) = \pi^{\ALPHABET M}_t(\EST(h_t)).
\]

\subsubsection{Sub-optimality bound}
We have assumed that the abstract model $\tilde {\ALPHABET M}$ satisfies Assumption~\ref{ass:meas-selection2}. We now show that Assumption~\ref{ass:lipschitz2} is satisfied.
\begin{lemma}\label{lem:mean_field}
Assumptions~\ref{assm:mean-field-dynamics} and~\ref{assm:mean-field-cost} imply Assumption~\ref{ass:lipschitz2} is satisfied with $F^c_t$ and $F^P_t$ defined as 
\[
    F_t^c(r) = L^{\bar \ell} r + 2\beta
    \quad\text{and}\quad
    F_t^P (r) = L^{\bar f} r + 2 \sum_{i=1}^n \alpha^i \gamma^i,
    \quad r \in \reals.
\]
\end{lemma}
See Appendix~\ref{app:mean_field_proof} for proof.

Furthermore, we have $\eta_t \le \sum_{i=1}^n \alpha^i r^i$ and, therefore, Assumption~\ref{ass:finite_eta2} is satisfied. Therefore, we can bound $\varepsilon_t$ and $\delta_t$ in Theorem~\ref{thm:main_theorem2} as
\[
   \varepsilon_t \leq L^{\bar \ell} \bar r + 2 \beta 
   \quad\text{and}\quad
   \delta_t \leq L^{\bar f} \bar r + 2 \bar \gamma
\]
where $\bar r = \sum_{i=1}^n \alpha^i r^i$ and $\bar \gamma = \sum_{i=1}^n \alpha^i \gamma^i$.

Therefore, we can upper bound the sub-optimality of using the certainty equivalent policy by the bound~\eqref{eq:main_theorem_rep_ex2} in Theorem~\ref{thm:main_theorem2}, which simplifies as follows:
\begin{align}\label{eq:lemma1_bound}
   &W^{\ALPHABET P,\mu^\EST}_t(h_t) - W^{\ALPHABET P}_t(h_t) \nonumber \\
   &\leq 2 \Bigl[ (T-t+1) (L^{\bar c}\bar r + 2\beta) + (L^{\bar f}\bar r + 2 \bar \gamma)\sum_{\tau = t}^{T-1} \Lip(V_{\tau+1}^{\ABSTRACT{\ALPHABET M}})\Bigr].
\end{align}
\begin{remark}
   Such models can arise in situations where there is a local controller associated with each particle, and the local controller ensures that the $\|f^i\|_{\infty} \leq \gamma ^i$. Further, note that the bound in~\eqref{eq:lemma1_bound} depends on $\bar r$, which may be small even if some of the $\{r^i\}_{i=1}^n$ are large.
\end{remark}

\section{Analysis}\label{sec:analysis}

In this section we provide the analysis for our main result. As stated earlier, Theorem~\ref{thm:main_theorem} is a special case of Theorem~\ref{thm:main_theorem2} obtained by taking $\ABSTRACT{\ALPHABET S} = \ALPHABET S$ and $\phi(s)=s$. Therefore, we only provide a proof of Theorem~\ref{thm:main_theorem2}.

We start with some background on policy independent beliefs,  and the AIS theory, followed by the key lemmas and proofs.

\subsection{Policy independent beliefs}\label{sec:pib}

Consider an arbitrary history dependent policy~$\mu$ for the model $\ALPHABET P$ defined in Sec.~\ref{sec:system_model}.
We define the following two \emph{beliefs} which are commonly used in POMDPs:
\begin{itemize}
   \item $b_{t|t}(\cdot | h_t)$ denotes the controller's posterior distribution on the current state $S_t$ given the history $h_t$ under the policy $\mu$, i.e., for any Borel subset $M_S$ of $\ALPHABET S$, $b_{t|t}(M_S |h_t) = \PR^{\mu}(S_t \in M_S | h_t)$. 
   The belief $b_{t|t}(\cdot | h_t)$ is referred to as the \emph{belief state}. It is well known that it does not depend on the choice of the history dependent policy~$\mu$~\cite{Astrom1965,Smallwood1973}.
   \item $b_{t+1|t}(\cdot, \cdot | h_t, a_t)$ denotes the controller's posterior distribution on the next state $S_{t+1}$ and next observation $Y_{t+1}$ given the history $h_t$ and action $a_t$ under policy $\mu$. Note that for any Borel subsets $M_S$ and $M_Y$ of $\ALPHABET S$ and $\ALPHABET Y$, we have
   \[
     \hskip -1.5em
      b_{t+1|t}(M_S, M_Y | h_t, a_t) = \int_{\ALPHABET S} P_t(M_S, M_Y | s_t, a_t) b_{t|t}(ds_t | h_t).
   \]
   Since the belief state $b_{t|t}(\cdot | h_t)$ does not depend on the choice of the policy~$\mu$, it follows from the above relationship that the same holds for $b_{t+1|t}(\cdot, \cdot | h_t, a_t)$ as well. With a slight abuse of notation, we will continue to use $b_{t+1|t}$ to denote its marginals on $\ALPHABET S$ or $\ALPHABET Y$.
\end{itemize}

\subsection{Approximate information states}\label{sec:AIS}

The AIS theory~\cite{AIS} provides a framework to derive sub-optimality bounds for a class of approximate solutions to POMDPs. The key idea in this framework is the notion of an approximate information state, which we formally define below. Our definition is similar to that of~\cite{AIS} with two differences. First, the analysis in~\cite{AIS} was done under the assumption that the state and observation spaces are finite, while we are working with Borel spaces. So, we include a \emph{measurable selection assumption} to ensure that the approximate dynamic program obtained from the AIS has a well-defined solution.  Second, the analysis in~\cite{AIS} used general integral probability metrics (IPMs)~\cite{muller1997}. We restrict our discussion to a specific choice of IPM (Wasserstein-1 distance), since that is the form that is used in our results. 

The discussion below is for the general POMDP model $\ALPHABET P$ defined in Sec.~\ref{sec:system_model}.
\begin{definition}
   Given sequences $\varepsilon = (\varepsilon_1, \dots, \varepsilon_T)$ and $\delta = (\delta_1, \dots, \delta_{T-1}) \in \reals^T_{\ge 0}$, 
   a process $\{Z_t\}_{t \ge 1}$, $Z_t \in \ALPHABET Z$, is called an $(\varepsilon, \delta)$-\emph{approximate information state (AIS)} if there exist
   \begin{itemize}
       \item a sequence of history compression functions $\{\sigma^{\AIS}_t\}_{t=1}^T$, where $\sigma^{\AIS}_t \colon \ALPHABET H_t \to \ALPHABET Z$ with $Z_t = \sigma^{\AIS}_t(H_t)$
       \item a sequence of cost approximators $\{c^{\AIS}_t\}_{t=1}^T$, where $c^{\AIS}_t \colon \ALPHABET Z \times \ALPHABET A \to \reals$
       \item a sequence of dynamics approximators $\{P^{\AIS}_t\}_{t =1}^{T-1}$, where $P^{\AIS}_t \colon \ALPHABET Z \times \ALPHABET A \to \Delta(\ALPHABET Z)$
   \end{itemize}
   such that following three properties are satisfied:
   \begin{enumerate}[leftmargin=3em]
       \item[(AP1)] \emph{Approximately sufficient for performance evaluation:}
       for any time $t \in \{1, \dots, T\}$ and any $h_t \in \ALPHABET H_t$ and $a_t \in \ALPHABET A$, we have
       \[
       \bigl| \EXP[ c_t(S_t, a_t) | h_t ] - c^{\AIS}_t(\sigma^{\AIS}_t(h_t), a_t) \bigr| \le \varepsilon_t
       \]
       \item[(AP2)] \emph{Approximately sufficient for predicting itself:} 
       for any time $t \in \{1, \dots, T-1\}$ and any $h_t \in \ALPHABET H_t$ and $a_t \in \ALPHABET A$, define the stochastic kernel $\nu_t$ on $\ALPHABET H_t \times \ALPHABET A_t \to \Delta(\ALPHABET Z)$ as follows: for any Borel measurable subset $M_Z$ of $\ALPHABET Z$, 
       \begin{align*}
           &\nu_t(M_Z | h_t, a_t) = \PR(Z_{t+1} \in M_Z | h_t, a_t) \\
           & = \int_{\ALPHABET Y} \IND\{ \sigma^{\AIS}_{t+1}(h_t,a_t,y_{t+1}) \in M_Z \} b_{t+1|t}(dy_{t+1} | h_t, a_t).
       \end{align*}
       Then, for any time $t \in \{1, \dots, T-1\}$, we have
       \[
       \dWas\bigl( \nu_t(\cdot | h_t, a_t), P^{\AIS}_t( \cdot | \sigma^{\AIS}_t(h_t), a_t) \bigr) \le \delta_t.
       \]
       \item[(M)]\emph{Measurable selection:} The MDP $\ALPHABET M^{\AIS} :=\langle \ALPHABET Z$, $\ALPHABET A$, $\{P^{\AIS}_t\}_{t=1}^{T-1}$, $\{c^\AIS_t\}_{t=1}^T, T \rangle$ satisfies measurable selection.
   \end{enumerate}
   The tuple $(\sigma^{\AIS}, c^{\AIS}, P^{\AIS})$, where each component is a sequence, is called an \emph{AIS-generator}. 
\end{definition}

We can  write a dynamic program for $\ALPHABET M^{\AIS}$  where the value functions $\{V^{\AIS}_t\}_{t=1}^{T+1}$, $V^{\AIS}_t \colon \ALPHABET Z \to \reals$, are defined as follows.
We initialize $V^{\AIS}_{T+1}(z) = 0$ for all $z \in \ALPHABET Z$ and then recursively define
for $t \in \{T, T-1,\dots, 1\}$
\begin{equation}\label{eq:ADP}
   V^{\AIS}_t(z_t) = \min_{a \in \ALPHABET A} \biggl\{
   c^{\AIS}_t(z_t, a) + \int_{\ALPHABET Z}V^{\AIS}_{t+1}(z') P^{\AIS}_t(dz'|z_t,a) 
   \biggr\}.
\end{equation}
The measurable selection condition (M) implies that there exists a measurable selector $\pi^{\AIS}_t \colon \ALPHABET Z \to \ALPHABET A$, $t \in \{1, \dots, T\}$, such that $\pi^{\AIS}_t(z_t)$ is an arg min of the right hand side of~\eqref{eq:ADP} and the functions $V^{\AIS}_t$ are measurable. From standard results in MDP theory~\cite{hernandez2012discrete}, we know that the policy $\pi^{\AIS} = (\pi^{\AIS}_1, \dots, \pi^{\AIS}_T)$ is an optimal policy for $\ALPHABET M^{\AIS}$.

The main result of the AIS theory is the following:
\begin{theorem}\label{thm:AIS}
   Define a history-dependent policy $\mu^{\AIS} = (\mu^{\AIS}_1, \dots, \mu^{\AIS}_T)$ for the POMDP $\ALPHABET P$ as follows: for any $t \in \{1, \dots, T\}$ and any $h_t \in \ALPHABET H_t$, define
   \[
   \mu^{\AIS}_t(h_t) = \pi^{\AIS}(\sigma^{\AIS}_t(h_t)).
   \]

   Then, for any $t \in \{1, \dots, T\}$ and $h_t \in \ALPHABET H_t$, we have
   \begin{equation}
       W^{\ALPHABET P, \mu^{\AIS}}_t(h_t) - W^{\ALPHABET{P}}_t(h_t) \le 2 \alpha_t
   \end{equation}
   where
   \[
       \alpha_t = \varepsilon_t + \sum_{\tau = t}^{T-1}\bigl[ \delta_{\tau} \Lip(V^{\AIS}_{\tau+1}) + \varepsilon_{\tau+1} \bigr].
   \]
\end{theorem}
\begin{proof}
   The result is the same as \cite[Theorem 9]{AIS}, which was stated under the assumption that $\ALPHABET S$ and $\ALPHABET A$ are finite sets while we are working with Borel spaces. As argued earlier, the measurable selection condition ensures that $V^{\AIS}_t$ and $\pi^{\AIS}_t$ are well-defined and measurable. Under this assumption, the approximation bound follows from exactly the same analysis as in \cite[Theorem 9]{AIS}.
\end{proof}

\subsection{Key lemmas}
The main idea of our sub-optimality bounds is to show that the abstract state estimation functions $\EST_t$, along with the per-step cost $\ABSTRACT c_t$ (defined in \eqref{eq:abstract-c}) and dynamics $\ABSTRACT P_t$ (defined in \eqref{eq:abstract-P}) of the abstract MDP $\ABSTRACT {\ALPHABET M}$ form an AIS generator for an appropriate choice of $\varepsilon$ and $\delta$.
$(\EST, \ABSTRACT c, \ABSTRACT P)$.

We first show that $\EST, \ABSTRACT c$ satisfy condition (AP1) of AIS.

\begin{lemma}\label{lem:P1}
   Under Assumptions~\ref{ass:lipschitz2} and~\ref{ass:finite_eta2},  for any  $h_t \in \ALPHABET H_t$ and $a_t \in \ALPHABET A$, we have
   \[
   \bigl| \EXP[ c_t(S_t, a_t) | h_t ] -  \ABSTRACT c_t(\EST_t(h_t), a_t) \bigr| \le F^c_t(\eta_t).
   \]
\end{lemma}
\begin{proof}
\begin{align}
\hskip 2em & \hskip -2em 
\bigl| \EXP[ c_t(S_t, a_t) | h_t ] -  \ABSTRACT c_t(\EST_t(h_t), a_t) \bigr|  \notag \\
&\le \EXP\bigl[ \bigl| c_t(S_t, a_t) - \ABSTRACT c_t(\EST_t(h_t), a_t)\bigr| \mid h_t \bigr]
\label{eq:c-1}
\end{align}
We now consider the inner term for a fixed realization $s_t$,
\begin{align}
&\lvert c_t(s_t, a_t) - \ABSTRACT c_t(\EST_t(h_t), a_t) \rvert \notag \\
&\stackrel{(a)}= \biggl| \int_{\phi^{-1}(\EST_t(h_t))}\bigl[c_t(s_t,a_t) -  c_t(s',a_t)\bigr]\lambda^c_t \bigl(ds'|\EST_t(h_t)\bigr) \biggr|
\nonumber \\
&\stackrel{(b)}\le \int_{\phi^{-1}(\EST_t(h_t))}\bigl|c_t(s_t,a_t) -  c_t(s',a_t)\bigr|\lambda^c_t \bigl(ds'|\EST_t(h_t)\bigr) 
\nonumber \\
&\stackrel{(c)}\le \int_{\phi^{-1}(\EST_t(h_t))}F^c_t(\dhSphi(s_t, s')) \lambda^c_t \bigl(ds'|\EST_t(h_t)\bigr)  
\notag \\
&\stackrel{(d)}=  F^c_t(\dhS(\phi(s_t), \EST_t(h_t))).
\label{eq:c-2}
\end{align}
where $(a)$ follows from definition of $\ABSTRACT c_t$, $(b)$ follows from Jensen's inequality, $(c)$ follows from Assumption~\ref{ass:lipschitz2} and $(d)$ follows from the fact that for any $s' \in \phi^{-1}(\EST_t(h_t))$, $\phi(s') = \EST_t(h_t)$ and that $\lambda^c_t(\phi^{-1}(\EST_t(h_t)) | \EST_t(h_t)) = 1$.  Substituting~\eqref{eq:c-2} in~\eqref{eq:c-1}, we get
\begin{align}
\hskip 2em & \hskip -2em 
\bigl| \EXP[ c_t(S_t, a_t) | h_t ] -  \ABSTRACT c_t(\EST_t(h_t), a_t) \bigr|  \notag \\
&\le \EXP[ F^c_t(\dhS(\phi(S_t), \EST_t(h_t))) | h_t ] \notag \\
&\stackrel{(e)}\le F^c_t\bigl(\EXP[ \dhS(\phi(S_t), \EST_t(h_t)) \mid h_t]\bigr) 
\notag \\
&\stackrel{(f)}\le F^c_t(\eta_t) 
\end{align}
where $(e)$ follows from Jensen's inequality and the concavity of $F^c_t$ and $(f)$ follows from the definition of $\eta_t$. 
\end{proof}

In order to show  that $\EST, \ABSTRACT P$ satisfy condition (AP2) of AIS, we will use the following intermediate lemma.
\begin{lemma}\label{lem:dWas}
Under Assumption~\ref{ass:lipschitz2}, for any $s_t \in \ALPHABET S$, $\ESTIMATE s_t \in \ABSTRACT{\ALPHABET S}$ and $a_t \in \ALPHABET A$, we have that
\begin{equation*}
       \dWas\bigl(P^{\phi}_{S,t}(\cdot|s_t,a_t), \ABSTRACT P_t (\cdot|\ESTIMATE s_t,a_t)\bigr) 
       \le F_t^P \Bigl(\dhS \bigl(\phi(s_t),\ESTIMATE s_t\bigr)\Bigr).
\end{equation*}
\end{lemma}
\begin{proof}
   \begin{align}
       &\dWas\bigl(P^{\phi}_{S,t}(\cdot|s_t,a_t), \ABSTRACT P_t (\cdot|\ESTIMATE s_t,a_t)\bigr) \nonumber \\
       &= \dWas\bigl(P^{\phi}_{S,t}(\cdot|s_t,a_t), \int_{\phi^{-1}(\ESTIMATE s_t)}P^{\phi}_{S,t}(\cdot|s',a_t)\lambda^P(ds'|\ESTIMATE s_t) \bigr) \nonumber \\
       &\stackrel{(a)}{\leq} \int_{\phi^{-1}(\ESTIMATE s_t)}\dWas\Bigl(P^{\phi}_{S,t}(\cdot|s_t,a_t), P^{\phi}_{S,t}(\cdot|s',a_t)\Bigr) \lambda^P(ds'|\ESTIMATE s_t) \nonumber \\
       &\stackrel{(b)}{\leq} \int_{\phi^{-1}(\ESTIMATE s_t)}F_t^P \Bigl(\dhS \bigl(\phi(s_t),\phi(s')\bigr)\Bigr)\lambda^P(ds'|\ESTIMATE s_t)  \nonumber \\
       &= F_t^P \Bigl(\dhS \bigl(\phi(s_t),\ESTIMATE s_t\bigr)\Bigr),
   \end{align}
   where $(a)$ follows from the convexity of Wasserstein distance~\cite[Thm. 4.8]{villani2009optimal} and $(b)$ follows from Assumption~\ref{ass:lipschitz2}.
\end{proof}

Next we define a stochastic kernel $\ESTIMATE \psi_t \colon \ALPHABET H_t \times \ALPHABET A_t \to \Delta(\ABSTRACT{\ALPHABET S})$, which is analogous to $\nu_t$ defined in (AP2).
For any $h_t \in \ALPHABET H_t$ and $a_t \in \ALPHABET A$ and  Borel measurable subset $M_{\ABSTRACT S}$ of $\ABSTRACT{\ALPHABET S}$, 
\begin{align}
   &\ESTIMATE \psi_{t}(M_{\ABSTRACT S} | h_t, a_t) = \PR(\EST_{t+1}(H_{t+1}) \in M_{\ABSTRACT S} | h_t, a_t) \nonumber \\
   & = \int_{\ALPHABET Y} \IND\{ \EST_{t+1}(h_t,a_t,y_{t+1}) \in M_{\ABSTRACT S} \} b_{t+1|t}(dy_{t+1} | h_t, a_t)\label{eq:hatnu}
\end{align}
which is the conditional probability distribution of $\ESTIMATE{S}_{t+1} =\EST_{t+1}(H_{t+1})$ given $h_t,a_t$. 

We also define the stochastic kernel $\ABSTRACT \psi_t \colon \ALPHABET H_t \times \ALPHABET A_t \to \Delta(\ABSTRACT{\ALPHABET S})$, which is used in the proof of the next lemma. 
For any $h_t \in \ALPHABET H_t$ and $a_t \in \ALPHABET A$ and  Borel measurable subset $M_{\ABSTRACT S}$ of $\ABSTRACT{\ALPHABET S}$, 
\begin{align}
    &\ABSTRACT \psi_{t}(M_{\ABSTRACT S} | h_t, a_t) = \PR(\ABSTRACT S_{t+1} \in M_{\ABSTRACT S} | h_t, a_t) \notag \\
   & = \int_{\ALPHABET S} \IND\{ \phi(s_{t+1}) \in M_{\ABSTRACT S} \} b_{t+1|t}(ds_{t+1} | h_t, a_t) \notag \\
   &= \int_{\ALPHABET S} P^{\phi}_{S,t}(M_{\ABSTRACT S} | s_t, a_t) b_{t|t}(ds_t | h_t, a_t),
   \label{eq:tilde-psi}
\end{align}
which is the conditional probability distribution of $\phi({S}_{t+1})$ given $h_t,a_t$.

The following lemma shows that $\EST_t$, $\ABSTRACT{P}_t$ satisfy (AP2).
\begin{lemma}\label{lem:P2}
  Under Assumptions~\ref{ass:lipschitz2} and~\ref{ass:finite_eta2}, for any  $h_t \in \ALPHABET H_t$ and $a_t \in \ALPHABET A$, we have
   \begin{equation*}
       \dWas(\hat\psi_{t}(\cdot | h_t, a_t), 
       \ABSTRACT P_{t}(\cdot | \EST_t(h_t), a_t) ) 
       \le F^P_t (\eta_t)+ \eta_{t+1},
   \end{equation*}
   where $\hat\psi_{t}(\cdot | h_t,a_t)$ is the probability distribution on $\ABSTRACT{\ALPHABET S}$ defined in \eqref{eq:hatnu}.
\end{lemma}
\begin{proof}
   Let $\ESTIMATE s_t = \EST_t(h_t)$. By triangle inequality, we have
   \begin{align}
       \hskip 2em & \hskip -2em
       \dWas(\ESTIMATE\psi_{t}(\cdot | h_t, a_t), 
       \ABSTRACT P_{t}(\cdot | \ESTIMATE s_t, a_t) ) \notag \\
      &\le 
       \dWas(\ESTIMATE\psi_{t}(\cdot | h_t, a_t), 
       \ABSTRACT\psi_{t}(\cdot | h_t, a_t))
       \notag \\
       &\quad + 
       \dWas(\ABSTRACT\psi_{t}(\cdot | h_t, a_t),
       \ABSTRACT P_{t}(\cdot | \ESTIMATE s_t, a_t) ). 
       \label{eq:two-terms}
   \end{align}
   Now we consider the two terms separately. The first term of~\eqref{eq:two-terms} can be bounded as follows.
   \begin{align}
       \hskip 2em & \hskip -2em
        \dWas(\ESTIMATE \psi_{t}(\cdot | h_t, a_t), 
       \ABSTRACT \psi_{t}(\cdot | h_t, a_t))
       \notag \\
       &\stackrel{(a)}= \inf_{ (\ESTIMATE{S}, \ABSTRACT S) \sim \Gamma(\ESTIMATE \psi_{t}, \ABSTRACT \psi_{t})}
       \EXP \bigl[
       \dhS(\ESTIMATE{S}, \ABSTRACT S) \bigr]
       \notag \\
       &\stackrel{(b)}\le \EXP[ \dhS\bigl(\EST_{t+1}(H_{t+1}), \phi(S_{t+1})\bigr) | h_t, a_t ] 
       \notag \\
       &= \EXP[ \EXP[ \dhS\bigl(\EST_{t+1}(H_{t+1}), \phi(S_{t+1})\bigr) | H_{t+1} ]  | h_t, a_t ] 
       \notag \\
       &\le \EXP[ \eta_{t+1}  | h_t, a_t ] 
       = \eta_{t+1}
       \label{eq:1st-term}
   \end{align}
   where, in $(a)$,  $\Gamma(\ESTIMATE \psi_{t}, \ABSTRACT \psi_{t})$ denotes the set of all joint measures with marginals $\ESTIMATE\psi_{t}(\cdot | h_t, a_t)$ and $\ABSTRACT \psi_{t}(\cdot | h_t, a_t)$, and in $(b)$, we use the fact that conditioned on $h_t,a_t$, the marginal distributions of $\EST_{t+1}(H_{t+1})$ and $\phi(S_{t+1})$ are   $\ESTIMATE \psi_{t}(\cdot | h_t, a_t)$ and $\ABSTRACT \psi_{t}(\cdot | h_t, a_t)$ respectively; therefore, the joint distribution on $(\EST_{t+1}(H_{t+1}), \phi(S_{t+1}))$ conditioned on $(h_t, a_t)$ lies in $\Gamma(\ESTIMATE \psi_{t}, \ABSTRACT \psi_{t})$.  

   The second term of~\eqref{eq:two-terms} can be bounded as follows.
   \begin{align}
       \hskip 2em & \hskip -2em
       \dWas(\tilde \psi_{t}(\cdot | h_t, a_t),
       \ABSTRACT P_{t}(\cdot | \ESTIMATE{s}_t, a_t) ) 
       \notag \\
       &= \dWas\Bigl( \int_{\ALPHABET S}P^{\phi}_{S,t}(\cdot|s_t,a_t)b_{t|t}(ds_t|h_t) ,
       \ABSTRACT P_{t}(\cdot | \ESTIMATE s_t, a_t) \Bigr)
       \notag \\
       &\stackrel{(a)}\le \int_{\ALPHABET S}  \dWas\bigl( P^{\phi}_{S,t}(\cdot|s_t,a_t) , \ABSTRACT P_{t}(\cdot | \ESTIMATE{s}_t, a_t) \bigr) b_{t|t}(ds_t|h_t)
       \notag \\
       &\stackrel{(b)}\le \int_{\ALPHABET S}  F^P_t \bigl(\dhS(\phi(s_t),\ESTIMATE{s}_t)\bigr)b_{t|t}(ds_t|h_t)\notag \\
       &= \EXP[ F^P_t\bigl(\dhS(\phi(S_t), \EST_t(h_t)) |  h_t ]\bigr) \notag \\
       &\stackrel{(c)}\le  F^P_t \bigl( \EXP\bigl[\dhS(\phi(S_t), \EST_t(h_t)) |  h_t \bigr]\bigr) \notag \\
       &\le F^P_t (\eta_t). \label{eq:2nd-term}
   \end{align}
   where $(a)$ follows from the convexity of Wasserstein distance~\cite[Thm. 4.8]{villani2009optimal},  and $(b)$ follows from Lemma~\ref{lem:dWas}, 
   and $(c)$ follows from Jensen's inequality and the concavity of $F_t^P.$
   \end{proof}
\subsection{Proof of Theorem~\ref{thm:main_theorem2}}\label{sec:thm_proof}
Under Assumptions~\ref{ass:lipschitz2} and~\ref{ass:finite_eta2}, Lemmas~\ref{lem:P1} and~\ref{lem:P2} ensure that conditions (AP1) and (AP2) of AIS are satisfied with $\varepsilon_t = F^c_t (\eta_t), \delta_t = F^P_t (\eta_t) + \eta_{t+1} $. Assumption~\ref{ass:meas-selection2} ensures that condition (M) of AIS is satisfied. Thus, the result follows from Theorem~\ref{thm:AIS}.

\section{Conclusion}\label{sec:conclusion}
In this paper, we introduced a generalization of the certainty equivalence principle for control policies in partially observable Markov decision processes (POMDPs). Our approach applies optimal state-feedback policies from the fully observable MDP to state estimates, without restricting to specific types of estimators such as MMSE. We established theoretical performance bounds that characterize their degree of sub-optimality. Specifically, we leveraged the approximate information state (AIS) framework~\cite{AIS} to quantify the impact of estimation errors on control performance, deriving bounds in terms of the smoothness of the system dynamics and the per-step cost function.

To illustrate the practical relevance of our results, we examined several examples that demonstrate that certainty equivalent policies can perform near-optimally when state estimation errors are small. This suggests that in scenarios where exact optimal policies are computationally intractable, certainty equivalent policies offer a practical and efficient alternative, making effective use of available state estimates to achieve reliable decision-making while maintaining tractability.

\bibliographystyle{IEEEtran}
\bibliography{references}

\appendices

\section{Proof of Lemma~\ref{lem:mean_field}}\label{app:mean_field_proof}

We prove the two parts separately. 

\subsubsection{Proof of~\eqref{eq:Lip-c-2}}
Arbitrarily pick $s, s' \in \ALPHABET S$. Let $m = \phi(s), m' = \phi(s') \in \ABSTRACT{\ALPHABET S} = \reals$. Then, by triangle inequality we have
\begin{align*}
|c(s,a) - c(s',a)| &\leq |\bar \ell(m,a) - \bar \ell(m',a)| + |\ell(s,a) - \ell(s',a)| \\
&\leq L^{\bar \ell} |m - m'| + 2\beta \coloneqq F^c_t(|m - m'|).
\end{align*}

\subsubsection{Proof of~\eqref{eq:Lip-P-2}}
Arbitrarily pick $(m,a) \in \ALPHABET S \times \ALPHABET A$. Let $M(w) = \bar f (m,a,w) + \sum_{i=1}^n\alpha^i f^i(m\textbf{1}_n,a,w)$ and $M'(w') = \bar f (m',a,w') + \sum_{=1}^n\alpha^i f^i(m'\textbf{1}_n,a,w')$. The left hand side of~\eqref{eq:Lip-P-2} is the Wasserstein distance between random variables $M(W)$ and $M'(W')$, where $W$ and $W'$ are identically distributed random variables with marginal distribution $P_W$. Let $\Gamma$ denote all joint couplings between $W$ and $W'$ such that the marginals are $P_W$. Then,
\begin{align*}
\dWas(M(W),M'(W')) &= \inf_{(W,W') \sim \Gamma} \EXP[ |M(W) - M'(W')|] \\
&\leq \EXP[|M(W)-M'(W)|] 
\end{align*}
where in the last equation we have chosen a specific coupling $W = W'$.

Now observe that
\begin{align*}
&\EXP[|M(W)-M'(W)|]  \\
&= \EXP\biggl[\biggl|\bar f(m,a,W) - \bar f(m',a,W) \\
&\quad + \sum_{i=1}^n \alpha^i\bigl(f^i(m\textbf{1}_n,a,W) - f^i(m'\textbf{1}_n,a,W) \bigr)\biggr|\biggr] \\
&\leq \EXP\bigl[|\bar f(m,a,W) - \bar f(m',a,W)|\bigr] \\ 
&\hspace{-1em}\quad + \sum_{i=1}^n \alpha^i\EXP\biggl[\biggl|\bigl(f^i(m\textbf{1}_n,a,W) - f^i(m'\textbf{1}_n,a,W)\biggr|\biggr] \\
&\leq L^{\bar f}|m-m'| + 2\sum_{i=1}^n\alpha^i \gamma^i 
\coloneqq F^P_t(|m-m'|)
\end{align*}
where the last inequality holds from Assumption~\ref{assm:mean-field-dynamics}.

\end{document}